\documentclass{amsart}
\usepackage{va}

\definecolor{amaranth}{rgb}{0.9, 0.17, 0.31}
\usepackage[colorlinks=true,citecolor=amaranth,linkcolor=black]{hyperref}%

\usepackage[normalem]{ulem}

\usepackage{todonotes}

\newcommand\tl{{}^t\kern -2pt}
\newcommand\C{\mathbb C}

\newcommand\R{\mathbb R}
\newcommand\Q{\mathbb Q}
\newcommand\Z{\mathbb Z}
\newcommand\ocK{\overline{\cK}}
\newcommand\NS{\operatorname{NS}}

\title{On lattice-polarized K3 surfaces}
\date{November 3, 2025}
\author{Valery Alexeev}
\email{valery@uga.edu}
\address{Department of Mathematics, University of Georgia, Athens, GA 30602, USA}

\author{Philip Engel} 
\email{pengel@uic.edu}
\address{Department of Mathematics, Statistics, 
and Computer Science, University of Illinois in Chicago,
Chicago, IL 60607, USA}

\begin{document}

\begin{abstract}
  We propose modifications to
  the commonly used definitions of 
  lattice-polarized and lattice-quasipolarized 
  \emph{smooth} K3 surfaces, collecting various
  versions of the definition, and determining the
  effects of these choices on the resulting
  moduli space. 
  We fill a gap in the theory, by
  replacing Weyl chambers with the
  new notion of a ``small cone'':
  the true datum in the definition
  of lattice quasipolarized 
  K3 surfaces.
  In addition, we describe
  the separated moduli stack and moduli space
  for lattice-polarized K3 surfaces with $ADE$ singularities, an important notion for applications.

\end{abstract}

\maketitle
\setcounter{tocdepth}{1}
\tableofcontents

\section{Introduction}

Lattice polarizations on K3 surfaces, 
generalizing ordinary
polarizations, were introduced and studied in an influential 1996 paper \cite{dolgachev1996mirror-symmetry} by Dolgachev. 
Unfortunately, Theorem~3.1 
describing the fibers
of the period map for lattice quasipolarized
K3 surfaces
in \emph{ibid.}~is incorrect with the
given definitions, 
and this 
misconception
has persisted throughout the 
literature to this day. In this paper, we give
the corrected definition (Definition~\ref{def:mqpol}), 
for lattice-quasipolarized K3 surfaces for which
\cite[Theorem~3.1]{dolgachev1996mirror-symmetry} 
remains true (Theorem~\ref{qpol-thm}). 

The difference between our definition and that of \cite{dolgachev1996mirror-symmetry} is
explained in Remark~\ref{rem:difference-between-defs}.
The two definitions 
agree over a Zariski open subset of the period space,
and their difference only becomes apparent at
non-separated points of moduli, where the
fiber of the period map jumps.

In \cite{dolgachev1996mirror-symmetry}, only
smooth K3 surfaces are considered. In Section~\ref{sec:ade-moduli}, 
we define lattice-polarized K3 surfaces
with ADE singularities (Definitions~\ref{def:lambda-pol1} and
\ref{def:lambda-pol2}). 

Furthermore, we describe the precise structure of the stacks  
both of the lattice-quasipolarized smooth K3 surfaces (Theorem~\ref{F-lambda-q}) and of the 
lattice-polarized $ADE$ K3 surfaces (Theorem~\ref{smooth-coarse} and Corollary~\ref{cor:F_Lambda for any h}).
Unlike the former, the latter stack is separated. The coarse moduli space of K3 surfaces with $ADE$ singularities is separated as well. This is the version of the moduli space that is useful for many applications, for example it was used in \cite{alexeev2023compact}. The (generalized) small cones 
introduced in Definitions~\ref{small-cone} and \ref{def:generalized small cone} are essential for this version of the moduli space.

We expect that our definition
of lattice (quasi)polarization
and the description of the stacks can be 
generalized to the setting 
of moduli of $K$-trivial varieties,
especially hyperk\"ahler varieties.

This material was included in 
\cite{alexeev2021compact}, 
an arXiv version of \cite{alexeev2023compact}, but not in the condensed, published version.
We have decided to make it widely available as a 
separate paper, to both clarify the problem 
and provide a solution to it.

Throughout the paper, we work over $\bC$.

\begin{acknowledgements}
  The authors were partially supported by the NSF, the first author under DMS-2201222 and the second author under DMS-2401104. 
\end{acknowledgements}

\section{Smooth vs ADE K3 surfaces}
\label{sec:smooth-vs-ADE}

\begin{definition}
  A {\it K3 surface} $X$ is a compact, complex, 
  smooth surface
  which has a trivial canonical bundle $\cO(K_X)\simeq\mathcal{O}_X$
  and irregularity zero, i.e.~$h^1(X, \cO_X)=0$.
\end{definition}

Any complex K3 surface is K\"ahler \cite{siu1983every-k3}
and any K3 surface admitting
a line bundle $L$ with $L^2>0$ is
algebraic and projective. 
When speaking of moduli of K3 surfaces, one
is always faced with a choice: to work with smooth surfaces or
with surfaces that can possibly have ADE (i.e.~Du Val, 
or ordinary double point) singularities. 
Both options have their advantages:
\smallskip

(1) Smooth K3 surfaces are all 
diffeomorphic and share the same
cohomology ring.
The price that one pays is
that instead of working with a \emph{polarization}, 
i.e.~an ample line bundle $L$, one is forced to 
work with a \emph{quasipolarization},
which is big, nef and semiample but whose 
intersection may be zero with finitely many 
smooth rational $(-2)$-curves on~$X$. An
unfortunate issue is that the moduli stacks 
and coarse moduli spaces of smooth quasipolarized 
K3 surfaces are not separated.\smallskip

(2) In contrast, when working with ADE K3 surfaces, 
one can work directly with polarized pairs $(\oX,\oL)$, 
where $\oL$ is an ample
line bundle. The moduli stacks and moduli 
spaces of polarized ADE surfaces 
are separated, at the cost of working
with singular varieties.
In other words,
the filling $(\overline{\cX},\overline{\cL})\to 
(C,0)$ of a family
of polarized ADE K3 surfaces 
$(\overline{\cX}\,\!^*, \overline{\cL}\,\!^*)\to C^*$
over a punctured curve is unique, if it exists.
Furthermore, the moduli stack is smooth, since
deformations of ADE K3 surfaces are unobstructed by \cite{burns1974local-contributions}.

\smallskip

For any individual polarized ADE surface $(\oX,\oL)$, 
the corresponding smooth
surface $X$ is the minimal resolution of 
singularities $\pi\colon X\to\oX$, and
$L = \pi^*(\oL)$. Vice versa, for a smooth quasipolarized surface
$(X,L)$, the line bundle $L$ is semiample and defines a contraction
$\pi\colon X\to\oX$ to its ADE companion, with $\oL = \pi_*(L)$. So,
the pairs $(X,L)$ and $(\oX,\oL)$ are in a natural bijection.

Only when working with families does the difference become
apparent, 
as was already noted in \cite{burns1975on-the-torelli}.
If 
$f\colon (\cX,\cL)\to (C,0)$ is a family of smooth
K3 surfaces for which the line bundle
$\cL_t$ is ample for a generic fiber $\cX_t$ but
$\cL_0$ is not ample on $\cX_0$, then
one can replace $f$ by a non-isomorphic family
$f^+\colon(\cX^+,\cL^+)\to (C,0)$ obtained from $\cX$ by a flop in a
$(-2)$-curve contained in the central fiber $\cX_0$. Fiberwise, the
two families are isomorphic, but not globally. As a consequence, the
moduli functor and the coarse moduli space of smooth quasipolarized
K3 surfaces are not separated.

However, the corresponding family
$\bar f\colon \overline{\cX}\to (C,0)$ of
ADE surfaces does not exhibit this behavior.
The extension over $0$ is unique:
It is obtained from $\cX$, resp.~$\cX^+$, 
by a contraction over $C$ defined by 
the semiample line bundle $\cL$, resp.~$\cL^+$. 

By \cite{brieskorn1970singular-elements}, 
a filling $(\overline{\cX},\overline{\cL})\to (C,0)$ 
of a family of polarized ADE K3 surfaces 
$(\overline{\cX}\,\!^*,\overline{\cL}\,\!^*)\to C^*$ admits, after a cyclic 
base change $(C',0)\to (C,0)$,
a simultaneous crepant resolution $(\cX,\cL)\to (C',0)$
 with $\cL$ relatively semiample.
The birational isomorphism $
\overline{\varphi}\colon \overline{\cX}\dashrightarrow 
\overline{\cX}\,\!^+$
between two such fillings 
extending ${\rm id}_{\overline{\cX}\,\!^*}$
induces then a birational isomorphism 
on the corresponding crepant resolutions 
$\varphi\colon \cX\dashrightarrow \cX^+$ 
over $C'$.
But then $\varphi^*(\cL^+)\sim_{C'} \cL$, so
$\varphi$ descends to
an isomorphism 
$\overline{\varphi}\,\!'\colon \overline{\cX}\,\!'
\to \overline{\cX}\,\!^{+}\,\!'$ of the relatively
ample models 
of $\cL$ and $\cL^+$
over $(C',0)$.
(For $f\colon \cX\to C$, the ample model of $\cL$ over $C$ is ${\Proj}_C \oplus_{d\ge0} f_*\cL^d$.)
These relatively ample models
$\overline{\cX}\,\!'$, $\overline{\cX}\,\!^{+}\,\!'$
are simply the base changes along $(C',0)\to (C,0)$
of $\overline{\cX}$ and $\overline{\cX}\,\!^+$.
Since $\overline{\varphi}'$ is an isomorphism, we deduce
that $\overline{\varphi}$ is too, as $\overline{\varphi}'$ is the base change of 
$\overline{\varphi}$.

\section{Analytic moduli}\label{analyticK3}

We begin by setting the
notation and reviewing fundamental results about analytic K3 surfaces.
For general references, see \cite{huybrechts2016lectures-on-k3} or 
\cite{asterisque1985geometrie-des-surfaces}.

The cup product endows $H^2(X,\Z)$ with a symmetric bilinear form 
$\cdot$, making it 
isometric to the unique even unimodular lattice of signature $(3,19)$. 
Let $\Omega\in H^0(X,K_X)$ denote a non-vanishing holomorphic 
$2$-form on $X$. Then there is a Hodge decomposition 
$$H^2(X,\C)=H^{2,0}(X)\oplus H^{1,1}(X)\oplus H^{0,2}(X)$$ 
with $H^{2,0}(X)=\C[\Omega]$ a complex line.
It is orthogonal with respect
to the Hodge form $(\alpha,\beta)\mapsto \alpha \cdot \overline{\beta}$.
By the Lefschetz $(1,1)$-theorem, the N\'eron--Severi group is
$$\NS(X)=H^{1,1}(X)\cap H^2(X,\Z)=[\Omega]^\perp\cap H^2(X,\Z).$$ 
Here, $H^{1,1}(X,\bR)\subset H^2(X,\bR)$ is 
$[\Omega]^\perp:=\langle \operatorname{Re}\Omega, 
\operatorname{Im}\Omega\rangle^\perp$.
Because $h^1(X)=0$ and $\pi_1(X)=1$,
we have $\NS(X)={\rm Pic}(X)$. 

\begin{definition}
  The {\it roots} of $X$ are the elements $\beta \in \NS(X)$ 
  for which $\beta^2=-2$. Associated to each root is a reflection
  $r_\beta\in O(H^2(X,\Z))$ defined by
  $$r_\beta(v) = v+(v\cdot \beta)\beta.$$ We define the {\it Weyl
    group} be the group
  $W_X:=\langle r_\beta \,:\, \beta\textrm{ a root of
  }X\rangle\subset O(H^2(X,\Z))$, 
  generated by reflections in the roots.
\end{definition}

\begin{remark}
  The group $W_X$ may be infinite because the roots $\beta$ lie in a
  vector space $H^{1,1}(X,\R)$ of hyperbolic signature $(1,19)$, and
  may be infinite in number.
\end{remark}

\begin{definition}
  The {\it K\"ahler cone} $\mathcal{K}_X\subset H^{1,1}(X,\R)$ is the
  set of classes of K\"ahler forms on $X$.  Denote by $\ocK_X\subset H^{1,1}(X,\R)$
  the closure of $\mathcal{K}_X$ in the set of positive norm vectors.
\end{definition}

Every K3 admits a K\"ahler form
  \cite{siu1983every-k3}. 
There are two connected components
$\mathcal{C}_X\sqcup (-\mathcal{C}_X)\subset H^{1,1}(X,\R)$ of
positive norm vectors, which are distinguished by the property
$\mathcal{K}_X\subset \mathcal{C}_X$. We call $\mathcal{C}_X$ the {\it
  positive cone}. Then $\ocK_X$ is a fundamental chamber for the
action of $W_X$ on $\mathcal{C}_X$.  We may now state the (first part
of the) Torelli theorem:

\begin{theorem}[\cite{piateski-shapiro1971torelli, burns1975on-the-torelli,
looijenga1980torelli}]\label{torelli-i}
  Two analytic K3 surfaces $X$ and $X'$ are isomorphic if and only if
  they are {\rm Hodge-isometric}, i.e.~there 
  exists an isometry
  $i : H^2(X',\Z)\to H^2(X,\Z)$ for which
  $i(H^{2,0}(X'))=H^{2,0}(X)$. 
  Furthermore, $i=f^*$ for an
  isomorphism $f:X\to X'$ if and only if
  $i(\mathcal{K}_{X'})=\mathcal{K}_X$. 
  This isomorphism $f$ is uniquely 
  determined by $i$.
\end{theorem}

Note that $\pm 1$ and $g\in W_X$ act by Hodge isometries on
$H^2(X,\Z)$. For any Hodge isometry $i$ between $X'$ and $X$, there is
a unique sign and unique element $g\in W_X$ such that
$\pm g\circ i(\mathcal{K}_{X'})=\mathcal{K}_X$ in which case, there is
an isomorphism $f:X\to X'$ satisfying $f^* = \pm g\circ i$. Thus, the
group of Hodge isometries of $X$ fits into a split exact sequence of groups
\begin{align}\label{exact}
    0\to \{\pm 1\}\times W_X\to {\rm HodgeIsom}(X)\to {\rm Aut}(X)\to 0.
\end{align}

\begin{definition}
Let $L_{K3}:=I\!I_{3,19}=U^{\oplus 3}\oplus E_8^{\oplus 2}$
be a fixed copy of the unique even unimodular lattice of
signature $(3,19)$, i.e.~the K3 lattice.
\end{definition}

Here, $U= I\!I_{1,1}$, and our convention 
is that $E_8$ is negative-definite.

\begin{definition}
Let $\pi \colon \mathcal{X}\to S$ be a family of
  smooth analytic K3 surfaces over an analytic space $S$. A {\it
    marking} is an isometry of local systems
  $\mu\colon R^2\pi_*\underline{\Z}\to \underline{L}_{K3}$.
\end{definition}

An analytic family $\mathcal{X}\to S$ admits a marking if and only if
the monodromy of the Gauss-Manin connection is trivial. If $S$ is a
point, then $\mathcal{X}=X$ is a K3 surface together with an isometry
$\mu\colon H^2(X,\Z)\to L_{K3}$. A marking over a single fiber
extends canonically to a marking over any simply connected base $S$.

\begin{definition}
  An {\it isomorphism of marked K3 surfaces}
  $f:(X,\mu)\to (X',\mu')$ is an isomorphism $f:X\to X'$ such
  that $\mu \circ f^* = \mu'$.
\end{definition}

Note that marked K3 surfaces have no nontrivial automorphisms: By
Theorem \ref{torelli-i}, the homomorphism
${\rm Aut}(X)\to O(H^2(X,\Z))$ is faithful.

\begin{definition} 
The {\it period domain} of analytic K3 surfaces is
  $$\mathbb{D}:=\mathbb{P}\{x\in L_{K3}\otimes \C\,:\, x\cdot x=0,
  \,x\cdot \overline{x}>0\}.$$ It is an analytic open subset of a
  $20$-dimensional quadric in $\mathbb{P}^{21}$. Let
  $(\mathcal{X}\to S, \,\mu)$ be a marked family of K3
  surfaces. The {\it period map} $P$ is defined by
  \begin{displaymath}
    P\colon S\to \mathbb{D}, \quad s\mapsto \mu(H^{2,0}(X_s)).
  \end{displaymath} 
\end{definition}

Since $H^0(X,T_X)=H^2(X,T_X)=0$, the deformation theory
of compact complex manifolds \cite{kodaira1986complex-manifolds}, 
\cite[Theorem 17]{kodaira1964on-the-structure}
implies the existence of a Kuranishi family:

\begin{theorem}
  Let $(X,\mu)$ be a marked K3 surface. There exists a family of K3
  surfaces $\mathcal{X}\to U$ over a complex analytic ball
  $U\subset \C^{20}$, which defines 
  a universal deformation of $X$, and
  such that the induced period map
  $P\colon U\to \mathbb{D}$ is an isomorphism onto its image.
\end{theorem}

Recall that an analytic 
deformation $\pi\colon \mathcal{X}\to (U,0)$
of the fiber $X=\pi^{-1}(0)$ is 
{\it universal} if it is the final
object in the category of germs of 
deformations of $X$. That is, given any flat 
family deforming $X$ over a base $(B,b)$,
there is a sufficiently small 
analytic neighborhood of the point
$b\in B$ over which $X$ fibers, 
such that the restricted family 
is the pullback
of $\pi$ along a unique map $(B,b)\to (U,0)$.

\begin{corollary}\label{per-iso}
  A family of marked K3 surfaces 
  $\mathcal{X}\to S$ over a smooth base
  $S$ is universal at every fiber $X_s$ where the differential $dP:T_sS\to T_{P(s)}\mathbb{D}$ of the
  period map is an isomorphism.
\end{corollary}

From these facts one may construct the moduli space of marked analytic
K3 surfaces by gluing together deformation spaces, 
as in \cite{burns1975on-the-torelli,
looijenga1980torelli, asterisque1985geometrie-des-surfaces}.

\begin{theorem}
\label{glue-def-smooth}
  There exists a fine moduli space $\cM$ of marked analytic 
  K3 surfaces, a non-separated, smooth complex
  analytic space
  of dimension $20$ endowed with an 
  \'etale period map $P\colon\cM\to \bD$, admitting a 
  universal family $\mathcal{X}\to \mathcal{M}$.
\end{theorem}

\begin{proof}[Sketch.]
  For every isomorphism type $(X,\mu)$ of marked K3 surface,
  construct a Kuranishi family
  $\mathcal{X}_{(X,\mu)}\to U_{(X,\mu)}$ over a contractible
  base, which is universal at all fibers (cf.~Corollary \ref{per-iso}), 
  and extend the marking
  $\mu$ on $X$ to the family. Then glue together these deformation
  spaces $$\mathcal{M}:=\bigcup_{(X,\mu)} U_{(X,\mu)}$$ by
  gluing $U_{(X,\mu)}$ and $U_{(X',\mu')}$ along the subset
  parameterizing isomorphic marked K3 surfaces. This subset is open
  because the family over $U_{(X,\mu)}$ is universal at every
  point, and the germ of a universal deformation is unique.

  Since marked K3 surfaces have no nontrivial automorphisms, there is
  a unique way to glue the two families
  $\mathcal{X}_{(X,\mu)}\to U_{(X,\mu)}$ and
  $\mathcal{X}_{(X',\mu')}\to U_{(X',\mu')}$ along this common
  open subset. So there exists a family $\mathcal{X}\to \mathcal{M}$
  of marked analytic K3 surfaces. Since we ranged over every
  isomorphism type and the universality holds at every point of
  $\mathcal{M}$, any marked family of analytic K3 surfaces is pulled
  back uniquely from $\mathcal{X}\to \mathcal{M}$. So $\mathcal{M}$ is
  a fine moduli space.
\end{proof}

The second part of the Torelli theorem is surjectivity of the period
map:

\begin{theorem}[\cite{todorov1980applications}]\label{torelli-ii}
  The period map $P\colon \mathcal{M}\to \mathbb{D}$ is surjective.
\end{theorem}

Let $x\in \mathbb{D}$ and define
$W_x:=\langle r_\beta\,:\, \beta\in 
x^\perp\cap L_{K3},\,\beta^2=-2\rangle$. 
If $P(X,\mu)=x$, then $W_X$ and $W_x$ are identified via
$\mu$, since 
${\rm NS}(X) = H^{2,0}(X)^\perp
\cap H^2(X,\bZ)$.
By Theorems \ref{torelli-i}, 
\ref{torelli-ii} and the description of the
K\"ahler cone, we have:

\begin{corollary}\label{cor:unpolarized-period-map}
  Let $x\in \mathbb{D}$ be a period. Then $P^{-1}(x)$ is a torsor over
  $\{\pm 1\}\times W_x$ with action given by
  $(X,\mu)\mapsto (X,g\circ \mu)$.
\end{corollary}

The moduli space $\mathcal{M}=\mathcal{M}_1\sqcup \mathcal{M}_2$
consists of two disjoint components, interchanged by
$\mu\mapsto -\mu$, because the monodromy 
of the universal family  does not permute
the decomposition of positive norm vectors in $H^{1,1}(X,\R)$
into components $\mathcal{C}_X\sqcup (-\mathcal{C}_X)$. 

\section{Quasipolarized moduli}\label{sec:qpol-moduli}

We now give analogous 
definitions to Section \ref{analyticK3}
in the
$\Lambda$-quasipolarized case. 
The standard reference here is
\cite{dolgachev1996mirror-symmetry}. 
However, Thm.~3.1 in \emph{ibid.}
is incorrect, see Remark~\ref{rem:dolgachev}. 
We give an alternative
definition of a lattice-polarized K3 surface for which
\cite[Thm.~3.1]{dolgachev1996mirror-symmetry} remains true.

Let $J\colon \Lambda\hookrightarrow L_{K3}$ be a fixed primitive, non-degenerate,
hyperbolic sublattice of signature $(1,r-1)$ with $r\leq 20$. 
The construction depends on this embedding 
and not just on the abstract lattice $\Lambda$.
We require the following notion:

\begin{definition}
A vector $h\in \Lambda\otimes \bR$ is {\it very irrational} if $h\notin
\Lambda'\otimes \bR$
for any proper,
saturated
sublattice $\Lambda'\subsetneq \Lambda$.
We will henceforth take $h\in \Lambda\otimes \bR$ 
to be a fixed very irrational vector of
positive norm $h^2>0$. 
\end{definition}

\begin{definition}\label{def:mqpol}
  A {\it $(\Lambda,h)$-quasipolarized K3 surface} 
  $(X,j)$ is a K3 surface
  $X$, and a primitive lattice embedding
  $j\colon \Lambda\hookrightarrow \NS(X)$ 
  for which $j(h)\in \ocK_X$ is big
  and nef. Two such $(X,j)$ and $(X',j')$ are {\it isomorphic} 
  if there is an isomorphism $f\colon 
  X\to X'$ of K3 surfaces for which
  $j=f^*\circ j'$.
\end{definition}

\begin{remark}\label{rem:difference-between-defs}
Consider the set of roots $\beta\in \Lambda$ and let
$W(\Lambda):=\langle r_\beta\rangle \subset O(\Lambda)$ 
be the corresponding reflection group. Then $h$ distinguishes 
a connected component $\cC_\Lambda$ of the positive
norm vectors, and an open cone 
$\cK_\Lambda\subset \mathcal{C}_\Lambda$
whose closure $\ocK_\Lambda$ is
a 
fundamental chamber for the action of 
$W(\Lambda)$ on $\cC_\Lambda$ that contains $h$.
The definition of a $\Lambda$-quasipolarization $j$ in \cite{dolgachev1996mirror-symmetry}
is that $j(\mathcal{K}_\Lambda)$ must contain
a big and nef class. This condition is implied by 
$j(h)$ being big and nef,
but is weaker.

Definition~\ref{def:mqpol} and Dolgachev's definition
agree when $\NS(X)=\Lambda$. In fact, Proposition
\ref{small-cone-prop} implies that the definitions agree on a finite
complement of Heegner divisors. But Remark \ref{rem:dolgachev}
gives an example in which the two definitions differ,
even for an ample (not just big and nef) quasipolarization.
\end{remark}

\begin{definition}
A {\it marking} of $(X,j)$ is an isometry
  $\mu:H^2(X,\Z)\to L_{K3}$ for which $J = \mu\circ j$.
\end{definition}

Since $j$ is determined by $\mu$ and $J$, 
it can be dropped from the
notation. As
$J\colon \Lambda\to L_{K3}$ is a
fixed embedding, we may identify 
$\Lambda\subset L_{K3}$ with a fixed
sublattice of the reference lattice 
$L_{K3}$. Then, ${\rm im}\,j$ is recovered
as $\mu^{-1}(\Lambda)$ and Definition
\ref{def:mqpol} says that
$\mu^{-1}(h)$ is big and nef.
Define the
$\Lambda$-quasipolarized period domain as
$$\mathbb{D}_\Lambda:=\mathbb{P}\{x\in \Lambda^\perp \otimes \C \,:\,x\cdot
x=0,\,x\cdot \overline{x}>0\}.$$ Since any element 
of the N\'eron--Severi
group $\NS(X)\subset H^2(X,\Z)$ is perpendicular 
to $H^{2,0}(X)$, we have that $\mathbb{D}_\Lambda$ 
contains the image of the period map on a
family of marked $\Lambda$-quasipolarized K3 surfaces
$(\mathcal{X}\to S ,\mu)$.

\begin{definition}
  Define the {\it Weyl group} of a point $x\in \mathbb{D}_\Lambda$ to be
  $$W_x(\Lambda^\perp):=\langle r_\beta\,:\, \beta \in x^\perp\cap
  \Lambda^\perp\rangle.$$
\end{definition}

Note that now $W_x(\Lambda^\perp)$ is finite because
$\Lambda^\perp\cap H^{1,1}(X,\bR)$ is negative-definite.

\begin{definition} Let $\underline{\mathcal{M}}_{\Lambda,h}$ 
be the moduli functor of marked 
$(\Lambda,h)$-quasipolarized K3 surfaces 
for a very irrational vector
$h\in\Lambda\otimes\bR$. 
\end{definition}

Theorem \ref{qpol-thm} shows that
even though the moduli functor 
$\underline{\mathcal{M}}_{\Lambda,h}$
depends on the very irrational vector 
$h$ chosen, the space representing it
doesn't. Even the universal families 
are the same for small deformations of $h$, 
by Remark \ref{big-on-small}.

\begin{theorem}\label{qpol-thm}
For any choice of a very irrational vector $h$, the moduli functor $\underline{\cM}_{\Lambda,h}$ is represented by a 
non-separated, smooth complex analytic space
$\cM_{\Lambda,h}$ of dimension $20-\rank\Lambda$, 
endowed with an \'etale, generically one-to-one period map 
$P_h\colon\mathcal{M}_{\Lambda,h}\to \bD_\Lambda$.
admitting a universal family 
$f_h\colon (\cX_h,\mu)\to \mathcal{M}_{\Lambda,h}$. 
For any $x\in\bD_\Lambda$, the fiber $P_h\inv(x)$ is a 
torsor over the finite group $W_x(\Lambda^\perp)$. 

For different very irrational vectors $h_1,h_2$, the moduli spaces $\cM_{\Lambda,h_1}$, $\cM_{\Lambda,h_2}$ are canonically isomorphic, so they can be identified with a fixed complex manifold,
which we denote $\cM_\Lambda$. 
\end{theorem}

The proof is a modification of that of
\cite[Thm. 3.1]{dolgachev1996mirror-symmetry}.

\begin{proof}
  First, consider the unpolarized period map for marked analytic K3
  surfaces $$P:\mathcal{M}\to \mathbb{D}$$ and take the sublocus $P^{-1}(\mathbb{D}_\Lambda)$. This is identified with the moduli space of marked K3 surfaces $(X,\mu)$ for which $\mu^{-1}(\Lambda)\subset {\rm Pic}(X)$. Let $x\in \mathbb{D}_\Lambda$ and choose $(X,\mu)\in P^{-1}(x)$ arbitrarily. The action of
  $g\in \{\pm 1\}\times W_x$ on the fiber is
  $(X,\mu)\mapsto (X,g\circ \mu).$ Since $\mu(\ocK_X)$ is a
  fundamental domain for the action of $\{\pm 1\}\times W_x$ on the positive norm vectors in $x^\perp\subset L_{K3}\otimes \bR$, there exists some $g$ for which $h\in (g\circ \mu)(\ocK_X)$.

  So we may as well have assumed $\mu$ itself satisfies
  $h\in \mu(\ocK_X)$. Next, we claim that the set of $g\in W_x$ for
  which $h\in (g\circ \mu)(\ocK_X)$ is equal to
  $W_x(\Lambda^\perp)\subset W_x$. Certainly any element $g\in W_x(\Lambda^\perp)$
  preserves this condition, because $g\vert_\Lambda={\rm id}_\Lambda$ and so
  $g(h)=h$. Conversely, since $W_x$ is a reflection group, any such
  element $g$ is generated by reflections in roots $\beta\in
  h^\perp$. 
  But the integral classes in $h^\perp$ are exactly the
  integral classes in $\Lambda^\perp$ because $h$ is very irrational. (This is the point of choosing a very irrational vector $h$.)
  So one has $g\in W_x(\Lambda^\perp)$.

    Using the \'etale map of Theorem~\ref{glue-def-smooth}, we identify an open subset $U_{x,h} \subset \cM_{\Lambda,h}$ with an open subset of $P\inv(\bD_\Lambda)$ in $\cM$. Gluing these neighborhoods for all $x\in\bD_\Lambda$, we build the fine moduli space $\cM_{\Lambda,h}$ as an open subset of $P\inv(\bD_\Lambda)$ together with a surjective period map to $\bD_\Lambda$.

For a vector $\beta\in L_{K3}$ one has $\beta\cdot x=0$ 
for all $x\in \bD_\Lambda$ if and only if $\beta\in\Lambda$.
Thus, for every root $\beta\in\Lambda^\perp$ the zero set 
$\beta^\perp\subset\bD_\Lambda$ is a proper closed subset. 
This implies that for any $x\in\bD_\Lambda$ 
outside of a 
locally finite countable union
of hyperplanes $\cup \beta^\perp$ with $\beta$ going over 
roots in $\Lambda^\perp$, one has $W_x(\Lambda^\perp)=1$. 
So, $P$ is generically one-to-one.

Let $h_1,h_2$ be two very irrational vectors. 
A point $y_1\in P_{h_1}\inv(x)$ corresponds to a
K3 surface $Y_1$ whose period is $x$. 
It defines a fundamental chamber $\fC_{y_1}$
in the vector space $x^\perp\cap(\Lambda^\perp\otimes~\bR)$ 
for the reflection group $W_x(\Lambda^\perp)$: 
the positive chamber for 
the positive roots on $Y_1$ which lie in $\Lambda^\perp$.
The choice of such a $W_x(\Lambda^\perp)$-chamber realizes the 
$W_x(\Lambda^\perp)$-torsor structure on $P_{h_1}\inv(x)$.

The same is true for the points $y_2\in P_{h_2}\inv(x)$.
Thus there is a canonical isomorphism of
the torsors $P_{h_1}\inv(x)\to P_{h_2}\inv(x)$
defined by sending $y_1\mapsto y_2$ if and only if 
$\fC_{y_1}=\fC_{y_2}$. This defines a canonical 
biholomorphism $\cM_{\Lambda,h_1}\to \cM_{\Lambda, h_2}$ 
and we deduce the independence of 
$\cM_{\Lambda,h}$ on~$h$ as a 
complex manifold.
\end{proof}

So $\underline{\mathcal{M}}_{\Lambda,h}$ is represented by a
submanifold $\mathcal{M}_{\Lambda,h}\subset \mathcal{M}$ of the space of
all marked K3 surfaces, and $\mathcal{M}_{\Lambda,h}\simeq \mathcal{M}_\Lambda$
for all very irrational $h$. Thus, 
we may make identifications
$\underline{\mathcal{M}}_{\Lambda,h}=\mathcal{M}_{\Lambda,h}=\mathcal{M}_\Lambda$
unless we need to either to distinguish the functors
$\underline{\mathcal{M}}_{\Lambda,h}$ for different~$h$, or distinguish the subloci
$\mathcal{M}_{\Lambda,h}\subset \mathcal{M}$.

\begin{corollary}\label{mpolcor} 
Let $\cM_{\Lambda,h}^\circ\subset \cM_{\Lambda,h}$ denote
the moduli space of marked $(\Lambda,h)$-quasipolarized
K3 surfaces $(X,\mu)$ for which $\mu(h)$ is ample. 
Then $P$ restricts to an isomorphism from 
$\cM_{\Lambda,h}^\circ$ to the complement of Heegner
divisors $$\bD_\Lambda^\circ:=
\{x\in \bD_\Lambda\,:\,x\cdot \beta\neq 0 
\textrm{ for all roots }\beta\in \Lambda^\perp\}.$$
\end{corollary}

The corollary is a modification of
\cite[Cor.~3.2]{dolgachev1996mirror-symmetry}.

\begin{proof} 
By definition, $\mu(h)$ is big and nef.
It is ample if and only if there does not exist a root
$\beta\in \NS(X)$ with $\beta\cdot \mu(h)=0$
Since $h$ is very irrational, we would necessarily have $\beta\in\Lambda^\perp$. 
Thus, $\mu(h)$ is ample exactly when
$W_x(\Lambda^\perp)=1$, and the corollary follows from Theorem \ref{qpol-thm}. \end{proof}

\begin{definition}\label{small-cone}
The {\it small cones} in $\cC_\Lambda\cup (-\cC_\Lambda)$,
which we will denote by $\sigma$,
are the connected components of the complements of the perpendiculars of every
vector $\beta\in L_{K3}\setminus \Lambda^\perp$ of square $\beta^2=-2$ for which
$\langle \Lambda,\beta\rangle\subset L_{K3}$ 
is hyperbolic, i.e.~it is nondegenerate 
of signature $(1,s)$ for some $s\in \bZ_{\geq 0}$.
\end{definition}

\begin{remark}
Definition \ref{small-cone}
is the key definition of this paper.
In \cite{dolgachev1996mirror-symmetry}, 
a cone $\cK_\Lambda$ is 
defined as an open Weyl chamber for the 
Weyl group $W(\Lambda)$ and a
``$\Lambda$-polarized quasi ample" K3 surface
is defined as a pair $(X,j)$ with an embedding
$j\colon \Lambda\hookrightarrow {\rm NS}(X)$ 
for which $j(\cK_\Lambda)$ contains 
a big and nef class.
This is the 
main source of trouble, since the walls of the 
K\"ahler cone $\cK_X$ are determined by all roots 
$\beta\in L_{K3}$, not just by those that 
lie in~$\Lambda$.
The cones $\cK_X$ are too 
``large'', in that there may be
too many isometries
$j\colon \Lambda \to  {\rm NS}(X)$ for which $j(\cK_\Lambda)$ contains a big and nef class,
leading in some cases to an infinite period
fiber, see Remark \ref{rem:dolgachev}.

Replacing $\cK_\Lambda$ 
with a small cone obviates the issue, 
doing what the choices of chambers were intended 
for in \cite{dolgachev1996mirror-symmetry};
see Proposition \ref{cone-exists} 
and Theorem \ref{qpol-thm}.
\end{remark}

\begin{remark}
    The fact that the definition of a lattice-polarized K3 surface is somewhat delicate and requires care was known to experts for a long time. For example, it was one reason for introducing the notion of a ``good embedding" in \cite[Section 4.2]{looijenga1984smoothing-components}.
    The issues with the original definition have been bypassed in different ways by some authors, for instance in \cite{joumaah2016non-symplectic} and \cite{camere2016lattice-polarized, camere2018some-remarks}, by introducing the notion of $\cK_\Lambda$-generality, employed e.g.~in \cite{boissiere2019complex-ball, 
brandhorst2023prime-order, 
bassi2025stable-cubic}.
\end{remark}

For the following examples, we note that any even 
lattice of rank $\le3$ admits a primitive embedding 
into $L_{K3}$, and if rank $\le2$ then this embedding 
is unique modulo $O(L_{K3})$ by a theorem of James \cite{james1972representations}, cf. \cite[Thm.~2.4]{looijenga1980torelli}.

\begin{example}\label{exa:small-cones1}
  Let $\Lambda=U(n)=\langle e,f\rangle$ with 
  $e^2=f^2=0$ and $e\cdot f=n>0$.
  Consider a lattice $\langle \Lambda, \beta\rangle$
  such that $\beta^2=-2$, $\beta\cdot e=a$,  
  $\beta\cdot f=-b$.  The hyperplane
  $\beta^\perp$ intersects the union $\cC_\Lambda\cup (-\mathcal{C}_\Lambda)$ of the
  first and third quadrants of $\Lambda$ iff $ab>0$.
If $\beta\not\in\Lambda$,
 this lattice is hyperbolic if and only if $ab<n$. 
The case $\beta\in\Lambda$ is possible only for $n=1$, then $ab=1$. 
  We see that in the case $\Lambda=U$ 
  (corresponding to elliptic
  K3 surfaces with a section) the small cone decomposition coincides
  with the Weyl chamber decomposition for $W(\Lambda)$. However, they are different for $n\ge2$. 
  In this case, there are only finitely many small cones.
  For $n\ge2$ the lattice $\Lambda$ has no roots, so $W(\Lambda)=1$ and $\cK_\Lambda=\cC_\Lambda$.
\end{example}

\begin{example}\label{exa:small-cones2}
  Similarly, let $\Lambda=\la e_1,e_2\ra$ be an even hyperbolic lattice with the quadratic form $q(x,y)=2n(x^2-dy^2)$, with $n,d\in\bN$ and non-square~$d$. 
   Consider a lattice $\langle \Lambda, \beta\rangle$
  such that $\beta^2=-2$, $\beta\cdot e_1=a$ and $\beta\cdot e_2=b$.
  The hyperplane $\beta^\perp$ intersects $\cC_\Lambda\cup (-\cC_\Lambda)$ iff $b^2-da^2>0$. If $\beta\not\in\Lambda$ then $\la\Lambda, \beta\ra$ is hyperbolic iff $b^2-da^2<4dn$. 
  The case $\beta\in\Lambda$ is possible only for $n=1$, in which case $b^2-da^2=4d$. 
  In all cases, 
  the infinitely many solutions of Pell's equations $b^2-da^2=k$ with $0<k<4dn$ give infinitely many walls of small cones. So, in this case there are infinitely many small cones. On the other hand, for $n\ge2$ the lattice $\Lambda$ has no roots, so $W(\Lambda)=1$ and $\cK_\Lambda=\cC_\Lambda$.

  The lattice vectors $\lambda\in\Lambda$ are in bijection with the algebraic integers $x+y\sqrt{d}\in\bZ[\sqrt{d}]$, and multiplying by a number with norm $N(x+y\sqrt{d})=x^2-dy^2=1$ is an isometry of $\Lambda$. This implies that modulo the isometry group $O(\Lambda)$ there are only finitely many small cones.
\end{example}

\begin{proposition}\label{small-cone-prop}
The small cone decomposition is locally rational polyhedral.
Let $G_\Lambda\subset O(L_{\rm K3})$
be the group of isometries of $L_{\rm K3}$ that preserve $\Lambda$.
The number of $G_\Lambda$-orbits of small cones is finite.
\end{proposition}

\begin{proof}
First, we show that the walls of the small cone 
decomposition are locally finite in $\cC_\Lambda\cup (-\cC_\Lambda)$. 
Write $\beta = \beta_\Lambda+\beta_{\Lambda^\perp}$ with
$\beta_\Lambda\in \Lambda\otimes \Q$ and $\beta_{\Lambda^\perp}\in \Lambda^\perp\otimes \Q$.
Note that $\beta_{\Lambda^\perp}$ lies in the hyperbolic space
$\langle \Lambda, \beta\rangle \otimes \Q$ and is perpendicular to $\Lambda$,
so we necessarily
have $(\beta_{\Lambda^\perp})^2\leq 0$. Thus, $(\beta_\Lambda)^2\geq -2$.

Next, observe that $\beta_\Lambda\in \Lambda^*$ lies in the
dual lattice, because intersecting with $\beta$
defines an integral linear functional on $\Lambda$. 
The hyperplane $(\beta_\Lambda)^\perp$ intersects the positive cone $\cC_\Lambda$ only if $\beta_\Lambda^2<0$. Since $\Lambda$ has a finite index in $\Lambda^*$, for the $\beta_\Lambda$ defining a wall, there are only finitely many possible values that $(\beta_\Lambda)^2$ can take, and the set of such vectors in $\Lambda^*$ is discrete.
The first statement of the proposition follows.

Now, observe that $G_\Lambda$ is a finite index subgroup
of $O(\Lambda)$. 
Let $\cP$  be a polyhedral 
fundamental domain with finitely 
many sides for the action of $G_\Lambda$ 
on the projectivization
of the positive
norm vectors $
\bH^{\rank\Lambda-1}\simeq 
\bP(\cC_\Lambda\cup (-\cC_\Lambda))=
\bP \cC_\Lambda$. 
It exists by 
\cite[Prop.~5.6]{bowditch1993geometrical-finiteness}.  

If $\rank \Lambda \ge3$ or $\Lambda$ is anisotropic
then this polyhedron $\cP$ has 
finite volume by a theorem of Borel and Harish-Chandra 
\cite[7.8, 9.4]{borel1962arithmetic-subgroups}. Adding finitely 
many cusps gives a compact space $\overline{\cP}$. 
The 
set of hyperplanes $\beta^\perp$ cuts it 
into finitely many cones.
This implies the second 
statement of the proposition.

In the case of $\rank\Lambda=2$ the cone $\cC_\Lambda$ is 
isomorphic to a quadrant in $\bR^2$ that is possibly 
divided into infinitely many small cones accumulating 
to the boundary. It is well known and easy to prove 
that the set of $O(\Lambda)$-orbits of vectors of a 
given norm in $\Lambda$ is finite. This implies that 
the set of the $O(\Lambda)$-orbits of the rays 
of these cones is finite, so there are only 
finitely many $O(\Lambda)$-orbits of the small cones.
\end{proof}

\begin{corollary}\label{cor:many-cones}
For any hyperbolic lattice $\Lambda$, if there is more than one small cone then there are infinitely many, unless $\Lambda$ is isotropic of rank~$2$.
\end{corollary}
\begin{proof}
By the proof of Proposition~\ref{small-cone-prop}, under these assumptions there is finite volume polyhedral fundamental domain $\cP$ for the action of $G_\Lambda$. If there is one wall $\beta^\perp$ separating two small cones then it intersects some polyhedron $g(\cP)$ in the $G_\Lambda$-orbit of $\cP$. 
Then there are infinitely many walls $g'(\beta^\perp)$ 
for $g'\in G_\Lambda$, so there are infinitely many small cones.
\end{proof}

On the other hand, if $\Lambda$ has no roots, 
then there is only one $W(\Lambda)$-Weyl chamber 
$\cK_\Lambda=\cC_\Lambda$. So, for many lattices 
$\Lambda$, a single Weyl chamber subdivides 
into infinitely many small cones.
On the opposite side of the spectrum, we have:

\begin{lemma}
    Let $\Lambda$ be  
    unimodular, 
    i.e.~isomorphic to $U$, $U\oplus E_8$ or 
    $U\oplus E_8^2$. Then the small cones 
    coincide with the Weyl chambers.
\end{lemma}
\begin{proof}
    Since $\Lambda$ is unimodular, it splits off as a direct summand: $\langle\Lambda,\beta\rangle = \Lambda\oplus B$. As in the proof of Proposition~\ref{small-cone-prop}, $\beta^\perp$ intersects the positive cone $\cC_\Lambda$ only if $(\beta_{\Lambda})^2<0$. Since $B$ is negative definite, one has $(\beta_{\Lambda^\perp})^2\le 0$. The equality $(\beta_{\Lambda^\perp})^2 + (\beta_{\Lambda})^2 = -2$ leaves the only possibility $\beta_\Lambda^2=-2$ and $(\beta_{\Lambda^\perp})^2= 0$, which implies $\beta_{\Lambda^\perp}=0$. So $\beta\in\Lambda$ and $\beta^\perp$ is a wall of a Weyl chamber.
\end{proof}


\begin{proposition}\label{cone-exists}
Let $X$ be a K3 surface, and
let $j:\Lambda\hookrightarrow {\rm Pic}(X)$ be an embedding. Then the following are equivalent:
\begin{enumerate} 
\item $j(h)\in \ocK_X$ i.e.~$j$ 
defines an $(\Lambda,h)$-quasipolarization.
\item $j(\sigma)\subset \ocK_X$ for the 
unique small cone $\sigma$ containing $h$.
\item\label{three} For any $h'\in\sigma$ 
in the same small cone
as $h$, $j(h')\in \ocK_X$.
\end{enumerate}
\end{proposition}

Observe that (\ref{three}) 
holds, whether or not $h'$ is very irrational.

\begin{proof}
The proposition follows from the statement 
that any wall of $\ocK_X$ is of the form 
$\beta^\perp$ for a root $\beta\in H^2(X,\bZ)$
and that the lattice 
$\langle j(\Lambda),\beta\rangle 
\subset \NS(X)\subset H^{1,1}(X,\bR)$
has hyperbolic signature for any root $\beta$. 
The walls of $\ocK_X$ associated
to roots $\beta\in j(\Lambda)^\perp$ 
do not decompose $\Lambda$.
\end{proof}

\begin{proposition}\label{same-functor}
  For two very irrational vectors $h_1,h_2$, the universal families $f_{h_1}\colon (\cX_{h_1},\mu)\to \mathcal{M}_{\Lambda}$ and
  $f_{h_2}\colon (\cX_{h_2},\mu)\to \mathcal{M}_{\Lambda}$  
  of marked K3 surfaces
  are isomorphic if and only if $h_1,h_2$ lie in the same small cone $\sigma$.  
\end{proposition}
\begin{proof}
  By Proposition \ref{cone-exists}, 
  $\cM_{h_1} \subset \cM$ and $\cM_{h_2}\subset\cM$ 
  are the same subset whenever 
  $h_1,h_2$ lie in the same small cone. 
  Conversely, suppose $h_1,h_2$
  lie in distinct small cones. Then there exists 
  a root $\beta\in L_{K3}$ for which 
  $\langle \Lambda,\beta\rangle \subset L_{K3}$ 
  is hyperbolic, $\beta\cdot h_1>0$, 
  and $\beta\cdot h_2<0$. By the hyperbolicity
  of $\langle \Lambda,\beta\rangle$, the 
  period domain $\bD_{\langle \Lambda,\beta\rangle}\subset \bD_\Lambda$ 
  is non-empty. So
  by Theorem \ref{qpol-thm}, there exists 
  a marked K3 surface $(X,\mu)\in \cM_{h_1}$
  for which the N\'eron--Severi group satifies
  $\langle \Lambda,\beta\rangle\subset \mu({\rm NS}(X))$.
  In particular, $\mu^{-1}(\beta)$ represents
  an effective class on $X$, and so
  $(X,\mu)\notin \cM_{h_2}$ as $\beta\cdot h_2<0$.
  Thus $\cM_{h_1}\neq \cM_{h_2}$ as subsets of $\cM$.
  
  Since $\cM$ is a 
  fine moduli space of marked K3 surfaces, 
  the statement follows.
\end{proof}

   As we see, the true datum in the definition of lattice quasipolarized K3 surfaces is a small cone $\sigma$, and an alternative notation for the moduli functor is $\underline{\cM}_{\Lambda, \sigma}$ for the moduli functor of marked $(\Lambda,\sigma)$-polarized K3 surfaces.

\begin{remark}
Since the number of $G_\Lambda$-orbits
of small cones is finite by Proposition \ref{small-cone-prop}, 
there is a sense in which the number of moduli 
spaces of $(\Lambda,\sigma)$-quasipolarized K3 surfaces 
for a fixed lattice $\Lambda$ is finite. 
An isometry $g\in G_\Lambda\subset O(L_{\rm K3})$ defines
a natural transformation
$\underline{\cM}_{\Lambda,\sigma}\to 
\underline{\cM}_{\Lambda,g(\sigma)}$
of moduli functors, given on objects
by $(X,\mu)\mapsto (X,g\circ \mu)$.  
\end{remark}

\begin{remark}\label{big-on-small}
A primitive integral vector $L$ of positive norm 
$L^2$, in the closure of $\sigma$, 
defines a big and nef line bundle on every
$(\Lambda,\sigma)$-quasipolarized K3 surface $(X,j)$. So any such $L$
defines an inclusion of $\mathcal{M}_{\Lambda,h}\hookrightarrow \mathcal{M}_{\Z L}$
by restriction $j\mapsto j\vert_{\Z L}$. Every $L\in \Lambda$ of 
positive square lies in the closure of at least one,
but possibly more than one, small cone.

If one defines $\underline{\cM}_{\Lambda,h}$ as in
Definition \ref{def:mqpol}, but allowing any 
(not necessarily very irrational)
choice of vector $h\in \Lambda\otimes \bR$, $h^2>0$, 
then  the space $\cM_{\Lambda,h}$ representing
this functor is still a 
(non-separated)
smooth complex analytic space
of dimension $20-r$,
as in Theorem \ref{qpol-thm}. But the fiber
of the period map is a finite 
union of $W_x(\Lambda^\perp)$-orbits, one for each 
small cone $\sigma$ containing $h$ in its closure.
\end{remark}

\begin{definition}
Let $\mathcal{F}_{\Lambda,\sigma}^{\rm q}$ denote the moduli stack of
$(\Lambda,\sigma)$-quasipolarized K3 surfaces, 
for a fixed small cone $\sigma$. 
Here ${\rm q}$
stands for ``quasipolarized.''
\end{definition}

Note that the families parameterized 
by $\mathcal{F}_{\Lambda,\sigma}^{\rm q}$ do not
have a marking.

\begin{theorem}\label{F-lambda-q}
There is an isomorphism of stacks
$\mathcal{F}_{\Lambda,\sigma}^{\rm q} = [\mathcal{M}_{\Lambda,\sigma}:\Gamma]$ where
$$\Gamma:=\{\gamma\in O(L_{K3})\,:\, \gamma\vert_\Lambda = {\rm id}_\Lambda \}$$
is the group of changes-of-marking. 
In particular, for different choices of the small cone $\sigma$ these stacks are naturally isomorphic.
The quotiented period map
$\mathcal{M}_{\Lambda,\sigma}/\Gamma\to \mathbb{D}_\Lambda/\Gamma$ is a bijection.
\end{theorem}

\begin{proof} The first part follows
formally from Theorem \ref{qpol-thm}; note that for two
small cones $\sigma_1, \sigma_2$ the 
canonical isomorphism $\cM_{\Lambda,\sigma_1}\to 
\cM_{\Lambda,\sigma_2}$ is $\Gamma$-equivariant. The second
part follows from the fact that $\Gamma\supset W_x(\Lambda^\perp)$
for any $x\in \bD_\Lambda$ and so $\Gamma$ acts transitively
on the fibers of the period map $P$.\end{proof}

\begin{remark}
  The action of $\Gamma$ on $\mathbb{D}_\Lambda$ (and thus also on
  $\mathcal{M}_\Lambda$) is properly discontinuous because $\mathbb{D}_\Lambda$ is
  a Type IV Hermitian symmetric domain for the group
  $\Gamma_\bR=O(2,20-r)$. So by the theorem of Baily-Borel
  \cite{baily1966compactification-of-arithmetic},
  $ \mathbb{D}_\Lambda/\Gamma$ is a quasiprojective variety. The period
  domain $\mathbb{D}_\Lambda$ has two connected components, interchanged by
  complex conjugation $x\mapsto \overline{x}$, and so the quotient by
  $\Gamma$ has either one or two connected components, depending on
  whether $\Gamma$ does or doesn't have an element interchanging the
  two components of $\mathbb{D}_\Lambda$.
\end{remark}

\begin{remark}\label{rem:dolgachev}
	
   As defined in \cite{dolgachev1996mirror-symmetry}, the space of marked
  ``$\Lambda$-polarized quasi ample" K3s, using our
  notations, is 
  \begin{displaymath}
     \cM_\Lambda' := 
     \bigcup_{h\in \cK_\Lambda} \cM_{\Lambda,h} =
     \bigcup_{\sigma\subset \cK_\Lambda} \cM_{\Lambda,\sigma}
     \subset \cM 
  \end{displaymath} 
  with $h$ ranging over all very irrational vectors, resp. with $\sigma$ ranging over all small cones
  contained in a fixed open Weyl chamber $\cK_\Lambda\subset\cC_\Lambda$ for $W(\Lambda)$. 
  Then \cite[Thm. 3.1]{dolgachev1996mirror-symmetry}
   claims that $P\inv(x)$ is a torsor over the finite group $W_x(\Lambda^\perp)$.
  In fact, with this definition, the fiber generally is not a torsor
  over any group, and it may be infinite 
  if $\cK_\Lambda$ contains infinitely many small cones $\sigma$.
  
  For example, start with a marked surface $(X,\mu)$ with a period $x$
  such that the lattice $\Pic(X)$ is reflective, e.g.~$U\oplus E_8$. 
  The chambers of the Weyl group $W_x$ cover
  the positive cone of $(U\oplus E_8)\otimes \bR$ and projectivize to
  hyperbolic polytopes of finite volume. 
  Let us assume that $\Lambda\subset U\oplus E_8$ is a lattice of rank $2$ and no roots, e.g. a lattice from Example~\ref{exa:small-cones2}. 
  Then the condition that $\cK_\Lambda=\cC_\Lambda$ contain
  a big and nef class is equivalent to the condition that
  the hyperbolic line $\bP \cC_\Lambda$ passes
  through the hyperbolic polytope 
  $$\mathcal{P}:=\mathbb{P}\mu(\ocK_X)
  \cap (U\oplus E_8)\otimes \bR.$$    
But any isometry $g \in W_x$ for which 
$g(\mathcal{P})\cap \mathbb{P}\cC_\Lambda\neq \emptyset$
  gives an ``$\Lambda$-polarized quasi-ample"
  K3 surface $(X,g\circ \mu)$ in
  the sense of \cite{dolgachev1996mirror-symmetry}.
Each of these chambers gives a point in the fiber 
of the period map. 

For an anisotropic $\Lambda$, the set of such chambers is infinite. 
  (Instead of a lattice or rank $2$, one could also take a lattice with $\rank\Lambda\ge3$ to the same effect, cf. Corollary~\ref{cor:many-cones}.)
And for a generic $\Lambda$, the set of such chambers is not naturally a torsor over any group. 
If, as in \cite{dolgachev1996mirror-symmetry}, one defines the moduli space of $\Lambda$-quasipolarized K3 surfaces to be $\cF'_\Lambda := \cM'_\Lambda / \Gamma$, then the map $\cM'/\Gamma\to \bD/\Gamma$ is not a bijection. Instead, for the lattices $\Lambda$ in this example, this map has infinite fibers.
 
We can furthermore
pick $\Lambda$ and $x$ 
in such a way that the Weyl group 
$W(x^\perp\cap \Lambda^\perp)$ is trivial. Then the
polarizations will be ample, not merely semiample.

For marked ``$\Lambda$-polarized ample K3 surfaces" 
as in \cite{dolgachev1996mirror-symmetry},
the period map $P$ is only an isomorphism onto 
its image on the sublocus
$\cM_\Lambda^{\circ\circ}\subset \cM_\Lambda^\circ \subset \cM_\Lambda$
where 
all roots $\beta\in \NS(X)$ 
satisfy $\beta^\perp\cap \cK_\Lambda=\emptyset$. 
The image
$P(\cM_\Lambda^{\circ\circ})$
is \begin{align*}\bD_\Lambda^{\circ\circ}:=
\{x\in \bD_\Lambda\,:\, &x\cdot \beta\neq 0 
\textrm{ for all roots }\beta\in L_{K3}
\textrm{ with } \\
&
\langle \Lambda,\beta\rangle\textrm{ hyperbolic and }(\beta_\Lambda)^2<0\}.\end{align*}
This locus is a complement of 
finitely
many $\Gamma$-orbits of Heegner divisors
in $\bD_\Lambda$ but it may contain the 
perpendiculars of vectors in $\Lambda^\perp$
of norm not equal to $-2$. 
As in the proof of Proposition \ref{small-cone-prop},
the finiteness of $\Gamma$-orbits
follows from the fact that 
$(\beta_{\Lambda^\perp})^2\geq -2$,
$\beta_{\Lambda^\perp}\subset (\Lambda^\perp)^*$,
and $\Lambda^\perp$ has finite 
index in $(\Lambda^\perp)^*$.
\end{remark}

\section{ADE moduli}\label{sec:ade-moduli}

We now describe a modification of the
above moduli problems which produces a separated moduli 
space and stack.
Theorem \ref{smooth-coarse} below is
known to experts, but to the authors' knowledge
there is no standard reference. 

The notion of ``lattice-polarized" is usually discussed only 
for smooth K3 surfaces as a special case of 
``lattice-quasipolarized", as in Corollary~\ref{mpolcor}.
where it gives an open subset of $\bD_\Lambda/\Gamma$, the complement of hyperplanes $\alpha^\perp$ for some some roots $\alpha$.
Here, instead we extend it to K3 surfaces with ADE singularities, and
the coarse moduli space becomes
the entire space $\bD_\Lambda/\Gamma$.

For us, a lattice-polarized ADE K3 surface 
is the {\it ample model} of a 
lattice-quasipolarized smooth K3 surface 
$(X,j)$, that is 
$\oX = \Proj \oplus_{d\ge0} H^0(X,L^d)$,
associated to a semiample line bundle $L\in j(\Lambda)$.

In Definition~\ref{def:lambda-pol1} and Theorem~\ref{smooth-coarse}, we begin with the case when $h$, the first Chern class of $L$, is generic, i.e.~lies in some (open) small cone $\sigma$. Then, in Definition~\ref{def:generalized small cone} and Corollary~\ref{cor:F_Lambda for any h}, we tackle the general case, when $h$ belongs to a face $\tau\subset\overline{\sigma}$ of a small cone, which we call a generalized small cone.

\begin{definition}\label{def:lambda-pol1}
Let $h\in \Lambda\otimes \bR$ be a 
vector contained in a
small cone $\sigma\subset \Lambda\otimes \bR$.

A {\it $(\Lambda,h)$-polarized K3 surface} 
$(\oX,j)$ is a
possibly singular surface $\oX$ with 
at worst rational double point (ADE) singularities,
whose minimal resolution $X\to \oX$ is a smooth
K3 surface, together with an isometric embedding
$j:\Lambda\hookrightarrow {\rm Pic}(\oX)$ for which
$j(h)$ is ample. 

A family of $(\Lambda,h)$-polarized K3 surfaces is a flat family $f\colon\overline{\cX}\to S$ of K3 surfaces with at worst ADE singularities together with a homomorphism $\Lambda\to \Pic_{\overline{\cX}/S}(S)$ such that every fiber is a $(\Lambda,h)$-polarized K3 surface with ADE singularities.
\end{definition}

For the same
reasons as Propositions \ref{cone-exists}, 
 \ref{same-functor}, the
functor of $(\Lambda,h)$- and 
$(\Lambda,h')$-quasipolarized K3 surfaces 
is the same, for all $h$, $h'$ in the same
small cone $\sigma$.
Thus, we may as well consider an
$h'\in\sigma$ which is {\it integral}.
Then, for a $(\Lambda,h)$-quasipolarized family
$\cX\to S$, the $S$-morphism 
$\cX\to \overline{\cX}$ defined by passing
to the ample model for $h'$
contracts all $(-2)$-curves in 
$\Lambda^\perp$ and produces a 
$(\Lambda,h)$-polarized
family.

\begin{remark}
In fact, if $h_1\in \sigma_1$ and $h_2\in \sigma_2$
lie in distinct small cones,
the $(\Lambda,h_1)$- and $(\Lambda,h_2)$-polarized K3 surfaces $X_1, X_2$
corresponding to a period $x\in \beta^\perp$, for $\beta$ defining
a wall separating $h_1$ and $h_2$,
are {\it not} isomorphic, since
$h_1$ is ample on $X_1$ but {\it not}
ample on $X_2$.  For this reason,
we cannot remove $h$ from the
definition.
\end{remark}

Consider a point $x\in\bD_\Lambda$. The stabilizer
$\Gamma_x\subset \Gamma$ of a point $x\in \mathbb{D}_\Lambda$
is the group of all Hodge isometries that fix the embedding $J:\Lambda\hookrightarrow L_{K3}$.
But the only Hodge isometries that are induced by
automorphisms are those that preserve the K\"ahler cone.
Thus there is an exact sequence of groups
$$0\to W_x(\Lambda^\perp)\to  \Gamma_x \to {\rm Aut}(\overline{X},j)\to 0$$
The inertia of the moduli stack of $(\Lambda,h)$-polarized K3 surfaces at the point
corresponding to $(\overline{X},j)$ is $\Gamma_x/W_x(\Lambda^\perp)$
whereas the inertia of the stack $[\mathbb{D}_\Lambda:\Gamma]$ at the point
$x$ is certainly $\Gamma_x$. 

To resolve this problem, 
consider an open neighborhood $U_x\ni x$ 
in $\mathbb{D}_\Lambda$
preserved by the stabilizer $\Gamma_x$. 
Take the quotient of
$U_x$ by $W_x(\Lambda^\perp)$. Since $W_x(\Lambda^\perp)\subset\Gamma_x$ is a normal
subgroup, the quotient group $\Gamma_x/W_x(\Lambda^\perp)={\rm Aut}(\overline{X},j)$
acts on the coarse space $U_x/W_x(\Lambda^\perp)$. 

\begin{definition}\label{colon-W}
Define DM stack 
$[\bD_\Lambda:_W\Gamma]$ by gluing local charts of the form
$[U_x/W_x(\Lambda^\perp):\Gamma_x/W_x(\Lambda^\perp)]$
ranging over all $\Gamma$-orbits 
$\Gamma\cdot x\subset \bD_\Lambda$.
\end{definition}

\begin{proposition} $[\bD_\Lambda:_W\Gamma]$ is 
a smooth and separated DM stack.  \end{proposition}

\begin{proof}
Since the quotient
$U_x/W_x(\Lambda^\perp)$ of a complex manifold 
by a reflection group is again
a complex manifold, we see that
$[\bD_\Lambda:_W \Gamma]$ is a smooth stack.
It is separated
because it factors the coarsening morphism
$[\bD_\Lambda\colon \Gamma]
\to [\bD_\Lambda:_W \Gamma] 
\to \bD_\Lambda/\Gamma$.
\end{proof}

\begin{theorem}\label{smooth-coarse}
For any $h$ contained in
a small cone $\sigma$,
the coarse moduli space $F_{\Lambda,h}$ of
$(\Lambda,h)$-polarized K3 surfaces is isomorphic to
$\mathbb{D}_\Lambda/\Gamma$, and the moduli
stack $\cF_{\Lambda,h}$ is
isomorphic to the separated DM
stack $[\bD_\Lambda:_W \Gamma]$.
In particular, both are independent of $h$.
\end{theorem}

\begin{proof}

Given any family $\pi\colon \mathcal{X}\to S$ of
$(\Lambda,h)$-quasipolarized K3 surfaces, 
there exists a simultaneous contraction
of the positive roots in $\Lambda^\perp$ on any fiber,
by taking 
$\cX\to \overline{\cX} = \Proj_S \oplus_{d\ge0} \pi_* \cL^d,$
for a relatively semiample element 
$\cL\in {\rm Pic}(\cX/S)$ that represents
an integral vector $h'\in \sigma$.
The result is a family of $(\Lambda,h)$-polarized K3
surfaces $\overline{\mathcal{X}}
\to S$.

Let $C\subset S$
be a curve, let $0\in C$ and $C^*:=C\setminus 0$. 
As in Section \ref{sec:smooth-vs-ADE}, 
the family $\overline{\mathcal{X}}$
depends only on the classifying morphism 
$C^*\to [\mathcal{M}_\Lambda:\Gamma]$
to the $(\Lambda,h)$-quasipolarized moduli stack
because we passed to the relatively ample model.
We deduce from Corollary \ref{F-lambda-q}
that the coarse moduli space $F_{\Lambda,h}$
is separated, and isomorphic to
$F_\Lambda \simeq (\cM_\Lambda/\Gamma)^{\rm sep}= 
\bD_\Lambda/\Gamma$.

To identify the stack structure on $\cF_{\Lambda,h}$
as $[\bD_\Lambda :_W \Gamma]$, 
we first prove that the Kuranishi space 
of $\oX$ is represented by the 
germ of $(\bD/W_x(\Lambda^\perp), \,x)$. 
This deformation theory statement follows from unobstructedness (\cite{burns1974local-contributions}, 
cf.~also \cite[Thm.~2.2]{namikawa2001deformation-theory})
together with 
\cite{brieskorn1970singular-elements},
which imply that the morphism of deformation
functors ${\rm Def}(X)\to 
{\rm Def}(\oX)$ given by simultaneous
contraction to a configuration of ADE
surface singularities $(\oX, \,p_i)\in \oX$ 
is a Weyl group cover 
$\prod_i W_{p_i} = W_x(\Lambda^\perp)$.

It follows that 
$\cF_{\Lambda,h}$ is locally identified
with the stack 
$$[\bD_\Lambda/W_x(\Lambda^\perp):\Gamma_x/W_x(\Lambda^\perp)]$$
defining the $:_W$-quotient. In terms of automorphism groups:
Any element of ${\rm Aut}(\oX,j)$ fixes $j(h)$ and thus
descends to the ample model for $h$, which is why 
the isotropy groups of $\cF_{\Lambda,h}$ 
are the same for any $h$.
\end{proof}

\begin{definition}\label{def:generalized small cone}
A {\it generalized small cone} 
$\tau$ is the relative interior
of a (locally) polyhedral face of some small cone $\sigma$,
which contains a positive norm vector.
\end{definition}

Note that every $h\in \Lambda\otimes \bR$, $h^2>0$,
lies in a unique generalized small cone $\tau$.

\begin{remark}
Removing the condition in Definition
\ref{def:lambda-pol1} that $h$ lies in a small
cone is subtle. In general,
the linear system defined by some 
integral vector in the generalized 
small cone $\tau\ni h$,
will contract $(-2)$-curves whose classes
do not lie in~$\Lambda^\perp$. Thus,
we do not always have an embedding of $\Lambda$
into ${\rm Pic}(\oX)$, for the ample model
$X\to \oX$ with respect to $h$.
\end{remark}

\begin{definition}\label{def:lambda-pol2}
Let $h\in \Lambda\otimes \bR$, $h^2>0$,
be {\it any} positive norm vector.
A family of {\it $(\Lambda,h)$-polarized 
K3 surfaces} is a flat family
$f\colon \overline{\cX}\to S$
of ADE K3 surfaces, together with
a simultaneous partial resolution
$\cX\to S$ to a family
of $(\Lambda,h')$-polarized K3
surfaces, where $h'\in \sigma$
lies in a specified small cone $\sigma$
containing $h$ 
in its closure. Here, by
a partial resolution $\cX\to \overline{\cX}$, 
we mean a proper birational contraction over $S$,
which fiberwise over $s\in S$ sits 
between $\overline{\cX}_s$ and 
its minimal resolution.
\end{definition}

\begin{theorem}\label{generalized-thm-lemma}
Let $\tau\ni h$ be the unique generalized
small cone containing a fixed vector
$h\in \Lambda\otimes \bR$, $h^2>0$.
The stack
of $(\Lambda,h)$-polarized K3 surfaces,
and the family $\overline{\cX}$ over it,
are independent of the choice of 
a small cone
$\sigma$ containing $\tau$ in its closure.
Furthermore, for all small cones
$\sigma$ containing $\tau$
in their closure, 
$\overline{\cX}$ admits a
simultaneous crepant resolution
to a family of $(\Lambda,\sigma)$-polarized K3 surfaces.
\end{theorem}

\begin{proof}
Let $h'\in \sigma$ be an element of the
chosen small cone $\sigma$ involved in 
Definition \ref{def:lambda-pol2}, of
the moduli stack of $(\Lambda,h)$-polarized
K3 surfaces.

We claim
there is an isomorphism of functors
$\cF_{\Lambda,h}\to \cF_{\Lambda,h'}$. On the one hand, by definition,
the universal family
over the former stack
is endowed with a simultaneous partial 
resolution to a family of 
$(\Lambda,h')$-polarized K3 surfaces, giving
the forward map. Conversely, there
is an inverse map $\cF_{\Lambda,h'}\to \cF_{\Lambda,h}$ by passing to the ample
model for some integral vector in the generalized
small cone $\tau\ni h$, which is relatively
semiample on any family of 
$(\Lambda,h')$-polarized
K3 surfaces.

By Theorem 
\ref{smooth-coarse}, for two small cones 
$\sigma_1$ and $\sigma_2$ containing $\tau$
in their closure, the moduli stacks $\cF_{\Lambda,\sigma_1}$
and $\cF_{\Lambda,\sigma_2}$ are 
naturally isomorphic.
Next, the $h$-ample models 
$\overline{\cX}_1$ and $\overline{\cX}_2$ 
of the universal families 
$\cX_{\sigma_1}\to \cF_{\Lambda,\sigma_1}$
and $\cX_{\sigma_2}\to \cF_{\Lambda,\sigma_2}$
for $\sigma_1$ and $\sigma_2$ 
are isomorphic, since $h$ is relatively ample
on these contractions. Thus
the universal family $\overline{\cX}_1\simeq 
\overline{\cX}_2\simeq \overline{\cX}$ is independent
of the choice of $\sigma$.
\end{proof}

\begin{corollary}\label{cor:F_Lambda for any h}
With respect to Definition
\ref{def:lambda-pol2}, the 
statement of Theorem \ref{smooth-coarse}
holds for any $h\in \Lambda\otimes \bR$, $h^2>0$.
\end{corollary}

\begin{remark}
The map $\cF^{\rm q}_{\Lambda,h}\to \cF_{\Lambda,h}$ 
is modeled on the following simple example. 

Let $\bA^1_t\cup \bA^1_t$ be a ``line with doubled origin", 
a non-separated scheme 
obtained by gluing two affine lines 
$\bA^1_t = \Spec \bC[t]$ along $\bA^1\setminus 0$, and let
$\bA^1_{t^2} = \Spec\bC[t^2]$. 
Over a transverse arc slicing the Heegner divisor
$\{x\in \bD_\Lambda\,:\,\beta\cdot x=0\}$, $\beta\in \Lambda^\perp$,
the stack 
$\cF^{\rm q}_{\Lambda,h}$ is locally modeled on 
$(\bA^1_t\cup \bA^1_t)/W$, $W=S_2=\langle r_\beta\rangle$, acting on each 
$\bA^1_t$ by $t\to -t$ and interchanging the two origins.
So the action of $W$ is free. This quotient is a non-separated 
algebraic space (it would be a non-separated stack if we
divided by a bigger group $\Gamma\supset W$). 
The stack $\cF_{\Lambda,h}$ in this case is modeled 
on the scheme $\bA^1_{t^2}$ and the map 
$\cF^{\rm q}_{\Lambda,h}\to\cF_{\Lambda,h}$ is modeled on 
\begin{displaymath}
    (\bA^1_t\cup\bA^1_t)/S_2 \to \bA^1_t/S_2 = \bA^1_{t^2}.
\end{displaymath}
As we see, the morphism $\cF_{\Lambda,h}^{\rm q}\to \cF_{\Lambda,h}$ is a ramified $W_t$-cover for all $t\in \bA^1_t$.
\end{remark}

Suppose that 
a non-separated DM stack $\cF$ admits charts 
$[U:G]$,
for which the union of charts $[U^{\rm sep}:G]$
is a separated DM stack $\cF^{\rm sep}$. 
The stack $\cF_{\Lambda,h}^{\rm q}$ satisfies this
condition, and with respect to this notion of separated
quotient, \cite[Thm.~2.11]{alexeev2023compact} is false, i.e.~$(\cF_{\Lambda,h}^{\rm q})^{\rm sep}\not\simeq \cF_{\Lambda,h}$.
Though \cite[Rem.~2.12]{alexeev2023compact}
describes the correct stack structure
on $\cF_{\Lambda,h}$ as in Definition \ref{colon-W} 
and Theorem \ref{smooth-coarse}.

It is unclear what is the most general set-up in which
one has the separated ``$:_W$-quotient'' that appears in
Theorem \ref{smooth-coarse}. We expect that
roughly, one needs a split exact sequence as in 
(\ref{exact})---that is, 
a normal subgroup of the relevant group
of Hodge isometries, which acts on the different
birational models which appear as
quasipolarized fillings 
of punctured families. 

For instance, the results in
this paper should extend to moduli
of polarized primitive symplectic varieties $(\oX,\oL)$, 
which play the role of ADE K3 surfaces,
and a $\bQ$-factorial terminalization $X\to \oX$,
which plays the role of the quasipolarized
minimal resolution, see \cite{bakker2022global-moduli} and 
\cite[Thm.~3.7]{lehn2024morrison-kawamata}. 

In this regard, we note some existing literature about the moduli spaces of smooth irreducible holomorphic symplectic 
manifolds:~\cite{camere2016lattice-polarized,camere2018some-remarks, joumaah2016non-symplectic}.

\bibliographystyle{amsalpha}

\begin{thebibliography}{LMP24}

\bibitem[AE21]{alexeev2021compact}
V. Alexeev and P. Engel, \emph{Compact moduli of {K}3 surfaces}, arXiv:2101.12186v2 (2021).

\bibitem[AE23]{alexeev2023compact}
\bysame, \emph{Compact moduli of {K}3 surfaces}, Ann. of Math. (2) \textbf{198} (2023), no.~2, 727--789. \MR{4635303}

\bibitem[Bas23]{bassi2025stable-cubic}
L.~L. Bassi, \emph{{GIT} stable cubic threefolds and certain fourfolds of {$K3^{[2]}$}-type}, arXiv:2301.11149 (2023).

\bibitem[BB66]{baily1966compactification-of-arithmetic}
W.~L. Baily, Jr. and A.~Borel, \emph{Compactification of arithmetic quotients of bounded symmetric domains}, Ann. of Math. (2) \textbf{84} (1966), 442--528. \MR{0216035 (35 \#6870)}

\bibitem[Bea85]{asterisque1985geometrie-des-surfaces}
A. Beauville et al,
\emph{G\'eom\'etrie des surfaces {$K3$}: modules et p\'eriodes}, Soci\'et\'e Math\'ematique de France, Paris, 1985, Papers from the seminar held in Palaiseau, October 1981--January 1982, Ast{\'e}risque No. 126 (1985). \MR{785216 (87h:32052)}

\bibitem[BC23]{brandhorst2023prime-order}
S. Brandhorst and A. Cattaneo, \emph{Prime order isometries of unimodular lattices and automorphisms of {IHS} manifolds}, Int. Math. Res. Not. IMRN (2023), no.~18, 15584--15638. \MR{4644970}

\bibitem[BCS19]{boissiere2019complex-ball}
S. Boissi\`ere, C. Camere, and A. Sarti, \emph{Complex ball quotients from manifolds of {$K3^{[n]}$}-type}, J. Pure Appl. Algebra \textbf{223} (2019), no.~3, 1123--1138. \MR{3862667}

\bibitem[BHC62]{borel1962arithmetic-subgroups}
A. Borel and Harish-Chandra, \emph{Arithmetic subgroups of algebraic groups}, Ann. of Math. (2) \textbf{75} (1962), 485--535. \MR{147566}

\bibitem[BL22]{bakker2022global-moduli}
B. Bakker and C. Lehn, \emph{The global moduli theory of symplectic varieties}, J. Reine Angew. Math. \textbf{790} (2022), 223--265. \MR{4472866}

\bibitem[Bow93]{bowditch1993geometrical-finiteness}
B.~H. Bowditch, \emph{Geometrical finiteness for hyperbolic groups}, J. Funct. Anal. \textbf{113} (1993), no.~2, 245--317. \MR{1218098}

\bibitem[BR75]{burns1975on-the-torelli}
D. Burns, Jr. and M. Rapoport, \emph{On the {T}orelli problem for k\"ahlerian {$K-3$} surfaces}, Ann. Sci. \'Ecole Norm. Sup. (4) \textbf{8} (1975), no.~2, 235--273. \MR{447635}

\bibitem[Bri71]{brieskorn1970singular-elements}
E.~Brieskorn, \emph{Singular elements of semi-simple algebraic groups}, Actes du {C}ongr\`es {I}nternational des {M}ath\'ematiciens ({N}ice, 1970), {T}ome 2, Gauthier-Villars, Paris, 1971, pp.~279--284. \MR{0437798}

\bibitem[BW74]{burns1974local-contributions}
D.~M. Burns, Jr. and J.~M. Wahl, \emph{Local contributions to global deformations of surfaces}, Invent. Math. \textbf{26} (1974), 67--88. \MR{349675}

\bibitem[Cam16]{camere2016lattice-polarized}
C. Camere, \emph{Lattice polarized irreducible holomorphic symplectic manifolds}, Ann. Inst. Fourier (Grenoble) \textbf{66} (2016), no.~2, 687--709. \MR{3477887}

\bibitem[Cam18]{camere2018some-remarks}
\bysame, \emph{Some remarks on moduli spaces of lattice polarized holomorphic symplectic manifolds}, Commun. Contemp. Math. \textbf{20} (2018), no.~4, 1750044, 29. \MR{3804263}

\bibitem[Dol96]{dolgachev1996mirror-symmetry}
I.~V. Dolgachev, \emph{Mirror symmetry for lattice polarized {$K3$} surfaces}, J. Math. Sci. \textbf{81} (1996), no.~3, 2599--2630, Algebraic geometry, 4. \MR{1420220 (97i:14024)}

\bibitem[Huy16]{huybrechts2016lectures-on-k3}
D. Huybrechts, \emph{Lectures on {K}3 surfaces}, Cambridge Studies in Advanced Mathematics, vol. 158, Cambridge University Press, Cambridge, 2016. \MR{3586372}

\bibitem[Jam72]{james1972representations}
D.~G. James, \emph{Representations by integral quadratic forms}, J. Number Theory \textbf{4} (1972), 321--329. \MR{309871}

\bibitem[Jou16]{joumaah2016non-symplectic}
M. Joumaah, \emph{Non-symplectic involutions of irreducible symplectic manifolds of {$K3^{[n]}$}-type}, Math. Z. \textbf{283} (2016), no.~3-4, 761--790. \MR{3519981}

\bibitem[Kod64]{kodaira1964on-the-structure}
K.~Kodaira, \emph{On the structure of compact complex analytic surfaces. {I}}, Amer. J. Math. \textbf{86} (1964), 751--798. \MR{187255}

\bibitem[Kod86]{kodaira1986complex-manifolds}
K. Kodaira, \emph{Complex manifolds and deformation of complex structures}, Grundlehren der mathematischen Wissenschaften [Fundamental Principles of Mathematical Sciences], vol. 283, Springer-Verlag, New York, 1986, Translated from the Japanese by Kazuo Akao, With an appendix by Daisuke Fujiwara. \MR{815922}

\bibitem[LMP24]{lehn2024morrison-kawamata}
C. Lehn, G. Mongardi, and G. Pacienza, \emph{The {M}orrison-{K}awamata cone conjecture for singular symplectic varieties}, Selecta Math. (N.S.) \textbf{30} (2024), no.~4, Paper No. 79, 36. \MR{4795880}

\bibitem[Loo84]{looijenga1984smoothing-components}
E.~Looijenga, \emph{The smoothing components of a triangle singularity. {II}}, Math. Ann. \textbf{269} (1984), no.~3, 357--387. \MR{761312}

\bibitem[LP81]{looijenga1980torelli}
E. Looijenga and C. Peters, \emph{Torelli theorems for {K}\"{a}hler {$K3$} surfaces}, Compositio Math. \textbf{42} (1980/81), no.~2, 145--186. \MR{596874}

\bibitem[Nam01]{namikawa2001deformation-theory}
Y. Namikawa, \emph{Deformation theory of singular symplectic {$n$}-folds}, Math. Ann. \textbf{319} (2001), no.~3, 597--623. \MR{1819886}

\bibitem[PSS71]{piateski-shapiro1971torelli}
I.~I. Pjatecki\u{\i}-Shapiro and I.~R. Shafarevi\v{c}, \emph{Torelli's theorem for algebraic surfaces of type {${\rm K}3$}}, Izv. Akad. Nauk SSSR Ser. Mat. \textbf{35} (1971), 530--572. \MR{0284440}

\bibitem[Siu83]{siu1983every-k3}
Y.~T. Siu, \emph{Every {$K3$} surface is {K}\"{a}hler}, Invent. Math. \textbf{73} (1983), no.~1, 139--150. \MR{707352}

\bibitem[Tod80]{todorov1980applications}
A.~N. Todorov, \emph{Applications of the {K}\"{a}hler-{E}instein-{C}alabi-{Y}au metric to moduli of {$K3$} surfaces}, Invent. Math. \textbf{61} (1980), no.~3, 251--265. \MR{592693}

\end{thebibliography}

\def\cprime{$'$}
\providecommand{\bysame}{\leavevmode\hbox to3em{\hrulefill}\thinspace}
\providecommand{\MR}{\relax\ifhmode\unskip\space\fi MR }
\providecommand{\MRhref}[2]{%
  \href{http://www.ams.org/mathscinet-getitem?mr=#1}{#2}
}

\end{document}